\documentclass{amsart}
\usepackage{hyperref}

\hypersetup{citecolor=blue,linkcolor=red}
\newtheorem{theorem}{Theorem}[section]

\newtheorem{lemma}[theorem]{Lemma}

\theoremstyle{definition}

\theoremstyle{remark}

\numberwithin{equation}{section}

\begin{document}

\title[]
 {Deforming a convex hypersurface with low entropy by its Gauss curvature}
\author[M.N. Ivaki]{Mohammad N. Ivaki}
\address{Institut f\"{u}r Diskrete Mathematik und Geometrie, Technische Universit\"{a}t Wien,
Wiedner Hauptstr. 8--10, 1040 Wien, Austria}
\curraddr{}
\email{mohammad.ivaki@tuwien.ac.at}

\dedicatory{}
\begin{abstract}
We prove the asymptotic roundness under normalized Gauss curvature flow provided entropy is initially small enough.
\end{abstract}

\maketitle
\section{Introduction}
Gauss curvature flow was proposed by Firey \cite{Firey} as the model for the wear of stones under tidal waves when $n=2.$ Assuming the existence, uniqueness, and regularity of the solutions (settled later by K.S. Chou (K. Tso) \cite{Tso}), he proved that if at the start the stone is centrally symmetric, then its ultimate shape is a unit ball. He conjectured that the same conclusion must hold if one starts the flow from any convex surface. Andrews gave an affirmative answer to this question \cite{Andrews1991}. To prove Firey's conjecture, Andrews made use of the parabolic maximum principle applied to the difference of principal curvatures and Chow's Harnack estimate (cf. \cite{Chow1991}) to obtain regularity of the normalized solution and the asymptotic roundness. In higher dimension, Guan and Ni by studying some fine properties of Firey's entropy functional and using Chow's Harnack inequality established the regularity of the normalized solutions without imposing any condition on initial smooth, strictly convex hypersurfaces. It then remains an interesting question to understand the nature of ultimate shapes. In this paper we prove the asymptotic roundness provided Firey's entropy functional is initially small. We prove a stability result for Firey's entropy functional and use some results of Guan-Ni \cite{GN} and a technique introduced by Andrews-McCoy in \cite{Andrewsjames2012} to show that if entropy is initially small, then the limiting shape must be "close" (depending only on the entropy of the initial body) to the unit ball in $C^k$-norm for any $k\ge0$. In particular, normalized hypersurfaces satisfy Chow's pinching condition \cite{Chow1985} for times close to the extinction time, and so they must flow to the unit ball.

\section{Results}
The setting of this paper is $n$-dimensional Euclidean space, $\mathbb{R}^{n}$. A compact convex subset of $\mathbb{R}^{n}$ with non-empty interior is called a \emph{convex body}. The set of convex bodies in $\mathbb{R}^{n}$ is denoted by $\mathcal{K}^n$. Write $\mathcal{F}^n$ for the set of smooth ($C^{\infty}$-smooth), strictly convex bodies in $\mathcal{K}^n$. The $n$-dimensional Lebesgue measure of $K\in\mathcal{K}^n$ is denoted by $V(K).$

The unit ball of $\mathbb{R}^n$ is denoted by $B$ and its boundary is denoted by $\mathbb{S}^{n-1}$. We write $\nu:\partial K\to \mathbb{S}^{n-1}$ for the Gauss map of $\partial K$, the boundary of $K\in\mathcal{F}^n$. That is, at each point $x\in\partial K$, $\nu(x)$ is the unit outwards normal at $x$.

The support function of $K\in \mathcal{K}^n$ as a function on the unit sphere is defined by
\[s_K(u):= \max_{x\in K} x\cdot u\]
for each $u\in\mathbb{S}^{n-1}$. Let $K$ be a smooth, strictly convex body and let $\varphi:M\to\mathbb{R}^n$ be a smooth parametrization of $\partial K$.
The support function of $K$ then can be expressed as
\[s_K(\nu(x)):=\varphi(x)\cdot \nu(x)\]
for each $x\in\partial K$. Write $\bar{g}_{ij}$ and $\bar{\nabla}$ for the standard metric and the standard Levi-Civita connection of $\mathbb{S}^{n-1}$. We denote the Gauss curvature of $\partial K$ by $\mathcal{K}$ and as a function on $\partial K$, it is related to the support function by \[\sigma_{n-1}:=\frac{1}{\mathcal{K}\circ\nu^{-1}}:=\frac{\det(\bar{\nabla}_i\bar{\nabla}_js+\bar{g}_{ij}s)}{\det \bar{g}_{ij}}.\]

Let $\varphi_0$ be a smooth, strictly convex embedding of $\partial K$. A family of convex bodies $\{K_t\}_t\subset \mathcal{F}^n$ defined by smooth embeddings $\varphi:M\times [0,T)\to \mathbb{R}^n$ is said to be a solution of the Gauss curvature flow if $\varphi$ satisfies the initial value problem
\begin{equation}\label{e: flow0}
 \partial_{t}\varphi(x,t)=-\mathcal{K}(x,t)\, \nu(x,t),~~
 \varphi(\cdot,0)=\varphi_{0}(\cdot).
\end{equation}
In this equation $\mathcal{K}(x,t)$ is the Gauss curvature of $\partial K_t:=\varphi(M,t)$ at the point where the outer unit normal is $\nu(x,t)$, and $T=\frac{V(K_0)}{nV(B)}$ is the maximal time that the flow exists. By Tso's work \cite{Tso} $K_t$ shrink to a point as $t\to T$.
It is easy to see that support functions of $\{K_t\}$ satisfy
\begin{align}\label{e: flow1}
 \partial_{t}s(z,t)=-\frac{1}{\sigma_{n-1}(z,t)},~~
 s(\cdot ,t)=s_{K_t}(\cdot).
\end{align}
P. Guan and L. Ni \cite{GN} studied the following entropy functional on the set of convex bodies:
\[\mathcal{E}(K)=\sup_{x\in \operatorname{int}K}\int_{\mathbb{S}^{n-1}}\log (s_K(u)-x\cdot u)d\theta.\]
There is a unique point $e(K)$ for which supremum is achieved, and if $V(K)=V(B)$ then $\mathcal{E}(K)\geq 0$, and the equality is achieved only for unit balls. Studying some of the fine properties of this functional played a key role in obtaining their $C^0$ and $C^2$ estimates for normalized solutions of the Gauss curvature flow which combined by Andrews's result in \cite{Andrews1997} led to the $C^{\infty}$ convergence of the normalized solutions to a smooth strictly convex, self-similar shrinker (e.q., a solution of $s=\frac{1}{\sigma_{n-1}}$) . It is an interesting question whether or not the ball is the unique ultimate shape of the Gauss curvature flow.

Let $K$ and $L$ be two convex bodies in $\mathbb{R}^{n}$ with respective support functions $s_K$ and $s_L$. The Hausdorff metric (cf. \cite[Lemma 1.8.14]{Schneider}) between two convex bodies $K$ and $L$ is defined as
\[\delta_{\mathcal{H}}(K,L)=\sup_{u\in\mathbb{S}^{n-1}}|s_K(u)-s_L(u)|.\]
Set $\tilde{K}:=\left(V(B)/V(K)\right)^{1/n}K.$ In this paper we first prove a stability result for the inequality $\mathcal{E}(\tilde{K})\geq 0:$
\begin{theorem}\label{thm: thm 1}
There exists $\gamma_n$, depending only on $n$, with the following property.
If $\mathcal{E}(\tilde{K})\leq\varepsilon$, then $\delta_{\mathcal{H}}(\tilde{K}-e(\tilde{K}),B)\leq \gamma_n\varepsilon^{\frac{1}{n+1}}$.
\end{theorem}
\begin{theorem}\label{thm: thm 2}
There exists $\varepsilon_0>0$, depending only on $n$, with the following property. If $K_t\in \mathcal{F}^n$ is a solution of the Gauss curvature flow with $\mathcal{E}(\tilde{K}_0)\leq\varepsilon\leq\varepsilon_0$ that shrinks to the origin, then there is a $0\leq\bar{t}<T$ such that $\|s_{\tilde{K}_{t}-e(\tilde{K}_t)}-1\|_{C^k}\leq f_{n,k}(\varepsilon)$ for all $t\ge \bar{t}$. Here $f_{n,k}\ge0$ is a continuous, non-decreasing function, independent of the solution, such that $f_{n,k}(\varepsilon)\to 0$ as $\varepsilon\to0.$
\end{theorem}
The next theorem gives the asymptotic roundness if the entropy is initially small.
\begin{theorem}\label{thm: thm 3}
Let $n\ge 4.$ There exists $\varepsilon_1>0$, depending only on $n$, with the following property. If $K_t\in \mathcal{F}^n$ is a solution of the Gauss curvature flow that shrinks to the origin and $\mathcal{E}(\tilde{K}_0)\leq\varepsilon_1$, then $\tilde{K}_t$ converges to the unit ball in $C^{\infty}$.
\end{theorem}
Under a certain pinching estimate, an upper bound on the ratio of the minimum and maximum principal curvatures of the initial hypersurface, convergence of the normalized solutions to the unit ball in $C^{\infty}$ was proved by B. Chow \cite{Chow1985}. Continuity of the functional $V(K)\exp(-\mathcal{E}(K)/V(B))$ in the Hausdorff distance shows that given an $\varepsilon>0$, we can find a smooth, strictly convex body with an arbitrarily large ratio of the minimum and maximum principal curvatures such that $V(K)\exp(-\mathcal{E}(K)/V(B))\ge V(B)/(1+\varepsilon)$ and thus $\mathcal{E}(\tilde{K})$ is still very small.\footnote{For example, $\tilde{K}$ can be obtained by cutting off negligible volumes from opposite caps of a ball and smoothing out the spherical edges and then normalizing the body to achieve the volume of the unit ball.} Therefore Theorem \ref{thm: thm 3} may be considered as a certain improvement on Chow's result.
\section{Proof of Theorem \ref{thm: thm 1}}
We start with the definition of $\delta_2$ metric of convex bodies.
The $\delta_2$ distance of $K$ and $L$ is defined by
\[\delta_2(K,L)^2=\|s_K-s_L\|_2^2:=\int_{\mathbb{S}^{n-1}}(s_K-s_L)^2d\theta.\]
Denote the diameter of a compact set $K$ by $D(K)$.
We recall \cite[Lemma 7.6.4]{Schneider} due to Vitale on the equivalency of $\delta_{\mathcal{H}}$ and $\delta_2$:
\begin{lemma}[Vitale, 1985]\label{lem: Vitale}
\[\delta_2(K,L)^2\geq \alpha_{n}D(K\cup L)^{1-n}\delta_{\mathcal{H}}(K,L)^{n+1},\]
where $\alpha_n=\frac{nV(B)\beta(3,n-1)}{\beta(1/2,(n-1)/2)}$ and $\beta(\cdot,\cdot)$ is the beta integral.
\end{lemma}
The following well-known entropy functionals on the set of convex bodies for $-n\leq p<0$ appear in connection to the $L_p$-Minkowski problems:
\[\mathcal{E}_{p}(K)=\inf_{x\in\operatorname{int}K}\int_{\mathbb{S}^{n-1}} (s_K-x\cdot u)^pd\theta.\]
The infimum are in fact achieved in interior points $e_{p}(K)$, see \cite{Ivaki2015}. Moreover, the Jensen, H\"{o}lder and Blaschke-Santal\'{o} inequality\footnote{The Blaschke-Santal\'{o} inequality states that $V(K)\inf\limits_{x\in\operatorname{int}K}\int_{\mathbb{S}^{n-1}} \frac{1}{(s_K(u)-x\cdot u)^n}d\theta\leq nV(B)^2.$ The unique point for which this infimum is achieved is called the Santal\'{o} point.} inequalities imply that if $V(K)=V(B)$, then
\begin{equation}\label{eq: rel entropy}
\exp(-\frac{\mathcal{E}(K)}{nV(B)})\leq \frac{\mathcal{E}_{-1}(K)}{nV(B)}\leq \sqrt{\frac{\mathcal{E}_{-2}(K)}{nV(B)}}\leq \cdots\leq \left(\frac{\mathcal{E}_{-n}(K)}{nV(B)}\right)^{\frac{1}{n}} \leq 1.
\end{equation}
\begin{lemma}\label{lema: stab} Let $0<\epsilon<1.$
Suppose $V(K)=V(B)$, and $\frac{\mathcal{E}_{-1}(K)}{nV(B)}\geq (1-\epsilon).$
Then
\[\delta_{\mathcal{H}}^{n+1}(K-e_{-2}(K),B_{r})
\leq \frac{2nV(B)}{\alpha_n}D(K)^2\left(D(K)+\frac{2}{1-\epsilon}\right)^{n-1}\epsilon,\]
for a constant $1\leq r\leq\frac{1}{1-\epsilon}$.
\end{lemma}
\begin{proof}
For simplicity, we denote the support function of $K-e_{-1}(K)$ by $s_{-1}$ and the support function of $K-e_{-2}(K)$ by $s_{-2}.$ Since $e_{-1}(K),e_{-2}(K)\in\operatorname{int} K$, $s_{-2}$ and $s_{-1}$ are both positive.
The following identity holds
\begin{equation}\label{eq: uu}\left|\left|\frac{\frac{1}{s_{-2}}}
{\left(\int_{\mathbb{S}^{n-1}}\frac{1}{s_{-2}^2}d\theta\right)^{\frac{1}{2}}}-\frac{1}{\left(\int_{\mathbb{S}^{n-1}}d\theta\right)^{\frac{1}{2}}}\right|\right|_2^2=2 \left(1-\frac{\int_{\mathbb{S}^{n-1}}\frac{1}{s_{-2}}d\theta}{\left(\int_{\mathbb{S}^{n-1}}
	\frac{1}{s_{-2}^2}d\theta\right)^{\frac{1}{2}}
	\left(\int_{\mathbb{S}^{n-1}}d\theta\right)^{\frac{1}{2}}}\right).
\end{equation}
Since by the assumption
\begin{equation}\label{eq: 5}
\int_{\mathbb{S}^{n-1}}\frac{1}{s_{-2}}d\theta\geq \int_{\mathbb{S}^{n-1}}\frac{1}{s_{-1}}d\theta\geq nV(B)(1-\epsilon),
\end{equation}
we obtain
\begin{align}\label{eq: u}
(1-\epsilon)\leq\frac{\int_{\mathbb{S}^{n-1}}\frac{1}{s_{-2}}d\theta}{\left(\int_{\mathbb{S}^{n-1}}\frac{1}{s_{-2}^2}d\theta\right)^{\frac{1}{2}}
\left(\int_{\mathbb{S}^{n-1}}d\theta\right)^{\frac{1}{2}}}\leq 1.
\end{align}
Rearranging terms in (\ref{eq: uu}) and using (\ref{eq: u}) yield
\begin{equation}\label{eq: u1}
\left|\left|s_{-2}-\frac{\left(\int_{\mathbb{S}^{n-1}}d\theta\right)^{\frac{1}{2}}}{\left(\int_{\mathbb{S}^{n-1}}\frac{1}{s_{-2}^2}d\theta\right)^{\frac{1}{2}}}
\right|\right|_2^2\leq  2nV(B)D(K)^2\epsilon.
\end{equation}
Define $r:=\frac{\left(\int_{\mathbb{S}^{n-1}}d\theta\right)^{\frac{1}{2}}}{\left(\int_{\mathbb{S}^{n-1}}\frac{1}{s_{-2}^2}d\theta\right)^{\frac{1}{2}}}.$ Note that \[nV(B)(1-\epsilon)^2\leq\int_{\mathbb{S}^{n-1}}\frac{1}{s_{-2}^2}d\theta\leq \int_{\mathbb{S}^{n-1}}\frac{1}{(s_K(u)-q\cdot u)^2}d\theta\leq nV(B),\] where $q$ is the Santal\'{o} point of $K$ and we used the H\"{o}lder and Blaschke-Santal\'{o} inequality to get the right-hand side inequality, and H\"{o}lder's inequality and inequality (\ref{eq: 5}) to get the left-hand side inequality. Hence we have a lower and upper bound on $r:$
\begin{align}\label{eq: u2}
	1\leq r\leq\frac{1}{1-\epsilon}.
\end{align}
Lemma \ref{lem: Vitale}, (\ref{eq: u1}), and (\ref{eq: u2}) imply the claim:
\begin{align*}
\delta_{\mathcal{H}}(K-e_{-2}(K),B_{r})^{n+1}&=\max_{\mathbb{S}^{n-1}}\left|s_{-2}-\frac{\left(\int_{\mathbb{S}^{n-1}}d\theta\right)^{\frac{1}{2}}}
{\left(\int_{\mathbb{S}^{n-1}}\frac{1}{s_{-2}^2}d\theta\right)^{\frac{1}{2}}}\right|^{n+1}
\\
&\leq \frac{2nV(B)}{\alpha_n}D(K)^2\left(D(K)+\frac{2}{1-\epsilon}\right)^{n-1}\epsilon.
\end{align*}
\end{proof}
\begin{lemma}\label{lema: stab2}
Suppose $V(K)=V(B)$. If
$\mathcal{E}(K)\leq \varepsilon,$
then
\begin{align*}
\delta_{\mathcal{H}}^{n+1}(K-e_{-2}(K),B_{r})\leq \frac{2nC_n^2(C_n+2)^{n-1}V(B)}{\alpha_n}\exp(\frac{(n+1)\varepsilon}{nV(B)})(1-\exp(-\frac{\varepsilon}{nV(B)})),
\end{align*}
for a constant $1\leq r\leq\exp(\frac{\varepsilon}{nV(B)})$ and a dimensional constant $C_n>0.$
\end{lemma}
\begin{proof}
By \cite[Corollary 2.1]{GN} we have $D(\tilde{K})\leq C_n\exp(\frac{\varepsilon}{nV(B)})$ for a dimensional constant $C_n>0.$ The claim then follows from inequality (\ref{eq: rel entropy}) and Lemma \ref{lema: stab} by setting $\epsilon=1-\exp(-\frac{\varepsilon}{nV(B)}).$
\end{proof}
Now we proceed to give the proof of Theorem \ref{thm: thm 1}.
We recall from \cite[Lemma 4.1]{GN} that
\[\mathcal{E}(\tilde{K})\geq \int_{\mathbb{S}^{n-1}} \log(s_{\tilde{K}}-e_{-2}(\tilde{K}))d\theta+C|e_{-2}(\tilde{K})-e(\tilde{K})|^2, \]
where $C$ depends only on $n$ and diameter for which we have $2\leq D(\tilde{K})\leq C_n\exp(\frac{\varepsilon}{nV(B)})$. On the other hand
\[\exp(-\frac{\int_{\mathbb{S}^{n-1}} \log(s_{\tilde{K}}-e_{-2}(\tilde{K}))d\theta}{nV(B)})\leq \left(\frac{\mathcal{E}_{-2}(\tilde{K})}{nV(B)}\right)^{\frac{1}{2}}\leq 1\Rightarrow \int_{\mathbb{S}^{n-1}} \log(s_{\tilde{K}}-e_{-2}(\tilde{K}))d\theta\geq 0.\]
Therefore \[\mathcal{E}(\tilde{K})\leq\varepsilon\Rightarrow|e_{-2}(\tilde{K})-e(\tilde{K})|^2\leq \frac{\varepsilon}{C}.\]
Combining this last inequality with Lemma \ref{lema: stab2} we infer that if $\mathcal{E}(\tilde{K})\leq\varepsilon$, then
\begin{align*}
\delta_{\mathcal{H}}(\tilde{K}&-e(\tilde{K}),B)\\
&\leq \delta_{\mathcal{H}}(\tilde{K}-e(\tilde{K}),\tilde{K}-e_{-2}(\tilde{K}))+\delta_{\mathcal{H}}(\tilde{K}-e_{-2}(\tilde{K}),B_r)+\delta_{\mathcal{H}}(B_r,B)\\
&\leq\sqrt{\frac{\varepsilon}{C}} +\delta_{\mathcal{H}}(\tilde{K}-e_{-2}(\tilde{K}),B_r)+ (\exp(\frac{\varepsilon}{nV(B)})-1).
\end{align*}
To complete the proof observe that the term $$\left(\frac{2nC_n^2(C_n+2)^{n-1}V(B)}{\alpha_n}\exp(\frac{(n+1)\varepsilon}{nV(B)})(1-\exp(-\frac{\varepsilon}{nV(B)}))\right)^{\frac{1}{n+1}}\sim o(\varepsilon^{\frac{1}{n+1}})$$ has the lowest order amongst all.
\section{proof of Theorem \ref{thm: thm 2}}
Write $r_-$ for the radius of the maximal ball enclosed by $K$ and  $r_+$ for the radius of the minimal ball enclosing $K$. We start with $0<\varepsilon_0\leq (\frac{1}{2\gamma_n})^{n+1}.$ By \cite[Theorem 3.4]{GN}, entropy is non-increasing along the normalized flow, so $\mathcal{E}(\tilde{K}_t)\leq \varepsilon_0.$

We apply the technique introduced by Andrews and McCoy in \cite{Andrewsjames2012}.
By Theorem \ref{thm: thm 1}, there exists $\varepsilon_1>0$ such that if $\mathcal{E}(\tilde{K}_0)\leq\varepsilon_0:=\min\{\varepsilon_1,(\frac{1}{2\gamma_n})^{n+1}\}$, then for all $t\ge0$ we have $\frac{r_+(\tilde{K}_t)}{r_-(\tilde{K}_t)}\leq 1+\eta=1+\frac{2\gamma_n\varepsilon^{\frac{1}{n+1}}}{1-\gamma_n\varepsilon^{\frac{1}{n+1}}}$, and $(1+2\eta)^{n}\leq\frac{5}{4}.$ Then the argument given in \cite[Section 12]{Andrewsjames2012} implies that for $t\geq \frac{3r_-(K_0)}{4n}$\footnote{Note that $T\geq \frac{r_-(K_0)^{n}}{n}$.} we have
$\mathcal{K}(T-t)^{\frac{n-1}{n}}\geq c>0 $ for a constant $c$ independent of the solution (That is, we get the same lower bound for any solution of Gauss curvature flow satisfying initially the entropy condition $\mathcal{E}\leq\varepsilon_0$.). Furthermore, an application of \cite[Theorem 6]{Andrews2000} applied to $\bar{s}^{t_{\ast}}(\cdot,t)=\frac{s(\cdot,t_{\ast}+(T-t_{\ast})t)}{(T-t_{\ast})^{1/n}}$ on $[-1,0]$ for any $t_{\ast}\ge T/2$ \footnote{$\bar{s}^{t_{\ast}}$ is a solution of the Gauss curvature flow (\ref{e: flow1}).} implies that
\[\mathcal{K}(T-t)^{\frac{n-1}{n}}\leq C<\infty  \]
for $t\ge T/2$ and for a constant $C$ independent of the solution.

In the rest of this section we assume that $\mathcal{E}(\tilde{K}_0)\leq\varepsilon_0$.
Consider $\bar{s}^{t_{\ast}}(\cdot,t)=\frac{s(\cdot,t_{\ast}+(T-t_{\ast})t)}{(T-t_{\ast})^{1/n}}$ on $[-1,0]$ for $t_{\ast}\in [\max\{3T/4,\frac{T+\frac{3r_-(K_0)^{n}}{4n}}{2}\},T).$ At time $t=-1$ the Gauss curvature has uniform lower and upper bounds. Moreover
\begin{itemize}
  \item  The minimum of the Gauss curvature is non-decreasing, \cite[Proposition 2.2]{Tso}.
  \item  The Gauss curvature remains bounded above as long as the $\bar{s}^{t_{\ast}}(\cdot,t)$ has a positive lower bound. The upper bound on the Gauss curvature depends only on the upper and lower bounds on $\bar{s}^{t_{\ast}}(\cdot,t),$ and the Gauss curvature of the convex body associated with $\bar{s}^{t_{\ast}}(\cdot,-1)$, \cite[Inequality (3.5)]{Tso}.
  \item  On $[-1,0]$, in view of $s_{\tilde{K}_t}=(n(T-t))^{-\frac{1}{n}}s_{K_t}$ and Theorem \ref{thm: thm 1} we get
  \begin{align*}
  \bar{s}^{t_{\ast}}(\cdot,t)\leq \bar{s}^{t_{\ast}}(\cdot,-1)=(2n)^{\frac{1}{n}}s_{\tilde{K}_{2t_{\ast}-T}}\leq (2n)^{\frac{1}{n}}(1+|e(\tilde{K}_{2t_{\ast}-T})|+\gamma_n\varepsilon^{\frac{1}{n+1}}),
  \end{align*}
  \begin{align*}
  n^{\frac{1}{n}}(1-|e(\tilde{K}_{t_{\ast}})|-\gamma_n\varepsilon^{\frac{1}{n+1}})\leq n^{\frac{1}{n}}s_{\tilde{K}_{t_{\ast}}}=\bar{s}^{t_{\ast}}(\cdot,0)\leq \bar{s}^{t_{\ast}}(\cdot,t).
  \end{align*}
\end{itemize}
Since $e(\tilde{K}_t)$ converges to the origin \cite{GN}, there exists $\zeta(K_0)$ such that for $t_{\ast}\in [\max\{3T/4,\frac{T+\frac{3r_-(K_0)^{n}}{4n}}{2},\zeta(K_0)\},T)$ we have $|e(\tilde{K}_{t_{\ast}})|\leq 1/4$ and so
   \begin{align*}
  n^{\frac{1}{n}}\frac{1}{4}\leq \bar{s}^{t_{\ast}}(\cdot,t)\leq (2n)^{\frac{1}{n}}\frac{7}{4}.
  \end{align*}
Therefore the Gauss curvature of $\bar{s}^{t_{\ast}}(\cdot,t)$ has uniform lower and upper bounds independent of $t_{\ast}\geq \max\{3T/4,\frac{T+\frac{3r_-(K_0)^{n}}{4n}}{2},\zeta(K_0)\}$ and $t\in[-1,0].$ Applying \cite[Theorem 10]{Andrews2000} then implies that
there are uniform lower and upper bounds on the principal curvatures of $\bar{s}^{t_{\ast}}(\cdot,t)$ on $t\in[-1/2,0]$ (these bounds are again independent of the solution $\tilde{K}_t$) and thus
$\bar{s}^{t_{\ast}}(\cdot,t)$ satisfies a uniformly parabolic equation on $[-1/2,0]$ with the parabolicity constant independent of $\tilde{K}_t.$ Uniform
bounds on all higher derivatives of the support functions now follow from a result of Krylov-Safanov \cite{KrylovSafonov} and Schauder estimates \cite[Theorem 4.9]{L}, see \cite[Theorem 4.1]{Tso} for details. In summary, if $\mathcal{E}(\tilde{K}_0)\leq \varepsilon\leq\varepsilon_0$ then for $t\geq\max\{3T/4,\frac{T+\frac{3r_-(K_0)^{n}}{4n}}{2},\zeta(K_0)\}$ and $k\ge1$
\[\sup_{\mathbb{S}^{n-1}}|s_{\tilde{K}_t-e(\tilde{K}_t)}-1|+\sum_{1\leq|m|\leq k}\sup_{\mathbb{S}^{n-1}}|\bar{\nabla}^ms_{\tilde{K}_t}|\leq C_k,\]
where $C_k>0$ are independent of $\tilde{K}_t$.
This, in particular, gives
\begin{align*}
\sup_{\mathbb{S}^{n-1}}|s_{\tilde{K}_{t}-e(\tilde{K}_t)}-1|+\sum_{1\leq|m|\leq k}\sup_{\mathbb{S}^{n-1}}|\bar{\nabla}^ms_{\tilde{K}_{t}-e(\tilde{K}_t)}|&\leq C_k'.
\end{align*}
We may then employ Theorem \ref{thm: thm 1}, the interpolation inequality and the Sobolev embedding theorem on $\mathbb{S}^{n-1}$ (e.g., Urbas \cite[Lemma 3.11]{Urbas}) to obtain
\[\|s_{\tilde{K}_t-e(\tilde{K}_t)}-1\|_{C^k}\leq f_{n,k}(\varepsilon)\]
for $t\geq\max\{3T/4,\frac{T+\frac{3r_-(K_0)^{n}}{4n}}{2},\zeta(K_0)\}.$ Here $f_{n,k}\ge0$ is some continuous, non-decreasing function, independent of $\tilde{K}_t$, such that $f_{n,k}\to0$ as $\varepsilon\to 0.$
\section{Proof of Theorem \ref{thm: thm 3}}
Chow \cite{Chow1985} proved that if the initial hypersurface $\partial\tilde{K}_0$ satisfies a certain pinching estimate, then $\tilde{K}_t$ converges to the unit ball in $C^{\infty}$. Choose $0<\varepsilon_1\leq\varepsilon_0$ small enough such that any smooth, strictly convex body with $\|s_{\tilde{K}-e(\tilde{K})}-1\|_{C^2}\leq f_{n,2}(\varepsilon_1)$ is forced to satisfy Chow's pinching assumption. Therefore in view of Theorem \ref{thm: thm 2}, if $\mathcal{E}(\tilde{K}_0)\leq\varepsilon_1$ then Chow's pinching assumption must hold for all $\tilde{K}_{t}$, $t$ close to $T$, so $\tilde{K}_t$ flows to the unit ball. $\Box$\\\\
\textbf{Acknowledgment}: The work of the author was supported by Austrian Science Fund (FWF) Project M1716-N25.
\bibliographystyle{amsplain}

\end{document}